\documentclass{amsart}

\usepackage{latexsym}
\usepackage{amsfonts,amssymb} 
\usepackage{graphicx}
\usepackage{tikz}
\usepackage{caption}
\usepackage{subcaption}

\newcommand{\field}[1]{\mathbb{#1}}
\newcommand{\A}{\field{A}}
\newcommand{\C}{\field{C}}

\newcommand{\N}{\field{N}}

\newcommand{\R}{\field{R}}
\newcommand{\Sp}{\field{S}}

\newcommand{\Z}{\field{Z}}

\theoremstyle{plain}

\newtheorem{theorem}{Theorem}[section]
\newtheorem{proposition}[theorem]{Proposition}
\newtheorem{lemma}[theorem]{Lemma}

\theoremstyle{definition}

\theoremstyle{remark}

\begin{document}

\makeatletter	   
\makeatother     

\title{On the Uniqueness of Polynomial Embeddings of the real 1-sphere in the plane}
\author{Gene Freudenburg}
\date{\today} 

\maketitle

\pagestyle{plain}

\begin{abstract} This paper considers real forms of closed algebraic $\C^*$-embeddings in $\C^2$. 
The classification of such embeddings was recently completed by Cassou-Nogues, Koras, Palka and Russell. Based on their classification, this paper shows that, up to an algebraic change of coordinates, there is only one polynomial embedding of the real 1-sphere $\Sp^1$ in the affine plane $\R^2$. 
\end{abstract}

\section{Introduction} 
Let $\Sp^n$ denote the real $n$-sphere as an algebraic variety over $\R$. Daigle asked whether every polynomial embedding of
$\Sp^1$ in $\R^2$ is equivalent to the standard embedding.\footnote{D. Daigle, University of Ottawa, private communication, 2013}
Our main result, {\it Theorem\,\ref{main}}, gives an affirmative answer to this question. 
This result mirrors the Epimorphism Theorem of Abhyankar and Moh, and Suzuki: Over a field $k$ of characteristic zero, any polynomial embedding of the affine line $\A^1_k$ in the affine plane $\A^2_k$ is equivalent to the standard embedding \cite{Abhyankar.Moh.75, Suzuki.74}.  
The complexification of $\Sp^1$ is the complex algebraic torus $\C^*$, and in contrast to its real counterpart, 
there are infinitely many equivalence classes of polynomial embeddings of $\C^*$ in $\C^2$. The proof of {\it Theorem\,\ref{main}}  relies on the recent classification of closed 
$\C^*$-embeddings in $\C^2$ due to Cassou-Nogues, Koras, Palka and Russell found in 
\cite{Cassou-Nogues.Koras.Russell.09, Koras.Palka.Russell.16, Koras.Palka.ppt}; see also \cite{Kaliman.96, Borodzik.Zoladek.10, Sathaye.11}. 

The proof of {\it Theorem\,\ref{main}} also uses the polar group of the real plane $\R^2$. The polar group of a real form of a complex affine variety is introduced in \cite{Freudenburg.ppt18a}. 

In their classification, Cassou-Nogues, Koras, Palka and Russell show that each equivalence class of closed embeddings of $\C^*$ in $\C^2$ is represented by a polynomial with rational coefficients. Therefore, every polynomial embedding of $\C^*$ in $\C^2$ admits a real form as an embedding. The proof of {\it Theorem\,\ref{main}} shows that, if a closed embedding of $\C^*$ in $\C^2$ admits two distinct real forms, then this embedding is equivalent to the standard embedding, given by $xy=1$. 

One is thus led to ask about polynomial embeddings of $\Sp^n$ in $\R^{n+1}$. To the author's knowledge, there are no known examples of such embeddings which are not equivalent. 

Similarly, we ask if there exists {\it any} polynomial embedding of the torus $\Sp^1\times\Sp^1$ in $\R^3$. The usual rendition of a topological torus as a surface of revolution in $\R^3$ does indeed give an algebraic surface $T$ which is diffeomorphic to $\Sp^1\times\Sp^1$. However, it turns out that $T$ is a nontrivial algebraic $\Sp^1$-bundle over $\Sp^1$. This is shown in {\it Section\,\ref{torus}}. Note that $\Sp^1\times\Sp^1$ is a real form of the complex torus 
$\C^*\times\C^*$, which has polynomial embeddings in $\C^3$, for example, $xyz=1$. 
\medskip

\noindent {\bf Notation and Terminology.} 
Let $R$ be a ring. 
The group of units of the ring $R$ is denoted $R^*$. 
The multiplicative monoid $R\setminus\{ 0\}$ is denoted $R'$.
The polynomial ring in $n$ variables over the ring $R$ is denoted $R^{[n]}$.

$\R^n$ denotes affine $n$-space over $\R$. The real $n$-sphere $\Sp^n$ is the algebraic variety in $\R^{n+1}$ defined by the polynomial equation $x_0^2+\cdots +x_n^2=1$. A {\bf polynomial embedding} of $\Sp^n$ in $\R^{N+1}$ is of the form $F=0$ for some $F\in \R [x_0,...,x_N]$, $N\ge n$.
 The {\bf standard embedding} is given by $x_0^2+\cdots +x_n^2=1$. Two embeddings are {\bf equivalent} if they differ by an algebraic automorphism of $\R^{N+1}$. 

Let $X$ be an affine $\R$-variety with coordinate ring $\R [X]$, and $Y$ an affine $\C$-variety with coordinate ring $\C [Y]$. 
Then $X$ is a {\bf real form} of $Y$ if $\C\otimes_{\R}\R[X]=\C [Y]$. 
\medskip

\noindent{\bf Acknowledgment.} The author wishes to acknowledge that many ideas in this paper were influenced by discussions with Daniel Daigle (University of Ottawa), Shulim Kaliman (University of Miami), Lucy Moser-Jauslin (Universite de Bourgogne), Peter Russell (McGill University) and Karol Palka (Warsaw University). 


\section{Preliminary Results} 

Throughout this section, $A$ is an affine integral domain over $\R$, and $B=\C\otimes_{\R}A=A[i]=A\oplus iA$. Assume that $B$ is also an integral domain. Given $f\in B$, write $f=f_1+if_2$ for $f_1,f_2\in A$. The {\bf conjugate} of $f$ is $\bar{f}=f_1-if_2$. 

\subsection{Polar Groups} Some facts about polar groups are required, as laid out in \cite{Freudenburg.ppt18a}. 

The element $f\in B'$ has {\bf no real divisor} if $f=rg$ for $r\in A$ and $g\in B$ implies $r\in A^*$. The set of $f\in B'$ with no real divisor is denoted $\Delta (B)$, and the set of irreducible elements of $\Delta (B)$ is denoted by $\Delta (B)_1$. 

\begin{theorem}\label{UFDUFD} {\rm (\cite{Freudenburg.ppt18a},Thm.\,5.1,Thm.\,5.5)} Assume that $A$ and $B$ are UFDs. Given $f\in B'$, 
$f\in\Delta (B)$ if and only if $\gcd (f,\bar{f})=1$. 
\end{theorem}

Let $K={\rm frac}(A)$ and $L={\rm frac}(B)$. The {\bf polar group} of $A$ is the quotient group $L^*/B^*K^*$, which is denoted $\Pi (A)$. 
This group is an invariant of $A$ which encodes information about the residual divisors in $B$ over $A$. 
Given $f\in B$, let $[f]$ denote its image in $\Pi (A)$. A key feature of this group is that $[f]^{-1}=[\bar{f}]$.

\subsection{Units and Gradings}

\begin{lemma}\label{trivial-units} Suppose that $A^*=\R^*$. If $f\in B^*$, then $f^{-1}=\lambda\bar{f}$ for some $\lambda\in\R^*$. 
\end{lemma}

\begin{proof} 
We have:
\[
ff^{-1}=1 \implies \bar{f}(\overline{f^{-1}})=1 \implies \bar{f}\in B^*
\]
Therefore, $f\bar{f}\in B^*\cap A=A^*=\R^*$. If $f\bar{f}=\rho$, then $\bar{f}=\rho f^{-1}$.
\end{proof}

\begin{lemma}\label{deg} Suppose that $B$ has a $\Z$-grading, and let $\deg$ be the induced degree function on $B$. 
If $A$ is a graded subring, then $\deg f=\deg\bar{f}$ for all $f\in B$. 
\end{lemma}

\begin{proof} Given $g\in B'$, let $\eta (g)$ denote the highest degree homogeneous summand of $g$.
Note that, since $A$ is a graded subring, $\eta (g)\in A$ if $g\in A$. 

Suppose that $f\in B'$, and write $f=f_1+if_2$ for $f_1,f_2\in A$. 
If $\deg f<\max\{\deg f_1,\deg f_2\}$, then 
$\deg f_1=\deg f_2$, which implies $\eta (f_1)+i\eta (f_2)=0$. But then $\eta (f_1)=\eta (f_2)=0$ implies $\eta (f)=0$, a contradiction. Therefore:
\[
\deg f=\max\{ \deg f_1,\deg f_2\} = \deg \bar{f}
\]
\end{proof}

Recall that an $\N$-grading of $B$ is a $\Z$-grading $\bigoplus_nB_n$ in which $B_n=\{0\}$ for $n\in\Z\setminus\N$.

\begin{lemma}\label{unit-degree} Suppose that $B$ has an $\N$-grading
and $A$ is a graded subring, and let $\deg$ be the induced degree function on $B$. 
Suppose that $P\in A'$ is prime in $B$ and $(A/PA)^*=\R^*$. Given $f\in B$, if the image of $f$ in $B/PB$ is a unit, then either $f\in B^*$ or 
$\deg P\le 2\deg f$.
\end{lemma}

\begin{proof} First note that $A/PA$ is a real form of $B/PB$. 

Assume that $f\not\in B^*$. Let $\pi :B\to B/PB$ be the standard surjection. By hypothesis, there exists $h\in B'$ such that $\pi (f)\pi (h)=1$. Therefore, there exists $Q\in B$ with $fh=1+PQ$. 
In addition, by {\it Lemma\,\ref{trivial-units}}, there exists $\lambda\in\R^*$ with $\pi (h)=\lambda \overline{\pi (f)}=\lambda\pi (\bar{f})$. Therefore, there exists $R\in B$ with 
$h=\lambda\bar{f}+PR$. So altogether we can write $\lambda f\bar{f}=1+PS$ for some $S\in B$. Note that $S\ne 0$, since $f\not\in B^*$. By {\it Lemma\,\ref{deg}}, we have $\deg f=\deg\bar{f}$. Therefore, 
$2\deg f=\deg P+\deg S\ge\deg P$.
\end{proof}

\subsection{Polynomial Rings}

In this section, assume that:
\[
A=\R [x,y]\cong\R^{[2]} \quad {\rm and}\quad B=\C\otimes_{\R}A=\C [x,y]\cong\C^{[2]}
\]
We consider the standard $\N$-grading of $A$ and $B$, wherein $x$ and $y$ are homogeneous of degree one. 

\begin{lemma}\label{AMS} Let $\alpha\in A$ be such that $B=\C [\alpha ,u]$ for some $u\in B$. Then there exists $\beta\in A$ such that $A=\R [\alpha ,\beta]$. 
If $u\in A$, then we may take $u=\beta$. 
\end{lemma}

\begin{proof} If $u\in A$, then $A=\R [\alpha ,u]$ by Cor.\,3.28 of \cite{Freudenburg.17}. So assume $u\not\in A$. 

$A/\alpha A$ is a real form of $B/\alpha B\cong\C^{[1]}$, and it is known that the only real form of $\C^{[1]}$ is $\R^{[1]}$ (see \cite{Russell.81}). Therefore, $A/\alpha A\cong\R^{[1]}$. By the Abhyankar-Moh-Suzuki Theorem \cite{Abhyankar.Moh.75, Suzuki.74}, there exists $\beta\in A$ with $A=\R [\alpha ,\beta ]$. 
\end{proof}

\begin{lemma} Let $Q=x^2+y^2-1\in A$. If $P\in A$ is such that $A/PA\cong_{\R}A/QA$, then:
\[
B/PB\cong_{\C}B/QB =\C[t,t^{-1}]
\]
\end{lemma}

\begin{proof} Let $A_1=A/QA$ and $A_2=A/PA$, and let $\alpha :A_1\to A_2$ be an isomorphism of $\R$-algebras. 
Let $B_1=B/QA=A_1[z]$, where $z^2+1=0$, and $B_2=B/PA=A_2[w]$, where $w^2+1=0$. Extend $\alpha$ to $\beta :B_1\to B_2$ by setting $\beta (z)=w$. 
Then $\beta$ is an $\R$-algebra isomorphism, and since $\beta (\R[z])=\R[w]$, we can view $\beta$ as an isomorphism of $\C$-algebras. 
\end{proof}

\begin{lemma}\label{P-form} Suppose that $u,v\in B$ satisfy $B=\C [u,v]$ and $[v]\ne 1$ in $\Pi (A)$. Let $P\in B'$ have the form $P=v^mf+1$ for $f\in B'$ and $m\ge 1$. Assume that:
\begin{enumerate}
\item $P\in A$
\item $P$ is irreducible in $B$ 
\item $(A/PA)^*=\R^*$
\end{enumerate}
Then $m=1$ and $fB=\bar{v}B$. 
\end{lemma}

\begin{proof} We have:
\[
P\in A \implies v^mf\in A \implies v^mf=\bar{v}^m\bar{f}
\]
Since $v$ is irreducible and $[v]\ne 1$, we see that $v\in\Delta (B)_1$. By {\it Thm.\,\ref{UFDUFD}}, $\gcd (v,\bar{v})=1$. Since $v$ and $\bar{v}$ are prime, it follows that $f\in\bar{v}^mB$. 
Write $f=\bar{v}^mg$ for $g\in B'$. Then $v^m\bar{v}^mg=\bar{v}^mv^m\bar{g}$ implies $g=\bar{g}$ and $g\in A$. 

Let $\pi :A\to A/PA$ be the standard surjection. Since $P=(v\bar{v})^mg+1$, we see that $\pi (g)$ is a unit of $A/PA$. By hypothesis, there exists $\lambda\in\R^*$ and $T\in A$ with 
$g=\lambda +PT$. If $T\ne 0$, then $\deg P=m\deg (v\bar{v}) + \deg P +\deg T$, which is not possible, since $\deg (v\bar{v})>0$. Therefore, $T=0$ and 
$g=\lambda\in A^*$, so $fB=\bar{v}^mB$. 
Let $\zeta =\lambda^{1/m}\in\C^*$. Then 
$P=(\zeta v\bar{v})^m+1$. Since $P$ is irreducible in $B$, $m=1$. 
\end{proof}

\begin{lemma}\label{cusp} Let $\tilde{B}=\C [t,t^{-1}]$. Suppose that $f^a=g^b$ for $f,g\in\tilde{B}$ and $a,b\in\N$ relatively prime. Then there exists $h\in\tilde{B}$ such that 
$f=h^b$ and $g=h^a$.
\end{lemma}

\begin{proof} If $a=1$ or $b=1$, this is clear, so we may assume that $b>a>1$. We first show that, for some $h\in\tilde{B}$:
\begin{equation}\label{equation1}
f^a\tilde{B}=g^b\tilde{B} \implies f\tilde{B}=h^b\tilde{B} \quad {\rm and}\quad g\tilde{B}=h^a\tilde{B}
\end{equation}
Let $f=dF$ and $g=dG$, where $d,F,G\in\tilde{B}$ and $\gcd (F,G)=1$. Then $F^a\tilde{B}=d^{b-a}G^b\tilde{B}$, and since $\gcd (F,G)=1$, we must have $G\in\tilde{B}^*$. 
Therefore, $F^a\tilde{B}=d^{b-a}\tilde{B}$.
By induction, we conclude that $F\tilde{B}=h^{b-a}\tilde{B}$ and $d\tilde{B}=h^a\tilde{B}$ for some $h\in\tilde{B}$. It follows that $f\tilde{B}=(d\tilde{B})(F\tilde{B})=(h^a\tilde{B})(h^{b-a}\tilde{B})=h^b\tilde{B}$. So the implication (\ref{equation1}) is proved.

Let $\omega ,\zeta\in\C^*$ and $m,n\in\Z$ be such that $f=\omega t^mh^b$ and $g=\zeta t^nh^a$. Then:
\[ 
f^a=g^b \implies am=bn \quad {\rm and}\quad \omega^a=\zeta^b \implies a\,\vert\, n \quad {\rm and}\quad b\,\vert\, m 
\]
Define $k=m/b=n/a$, and let $\lambda\in \C^*$ be such that $\omega =\lambda^b$ and $\zeta=\lambda^a$. If $H=\lambda t^kh$, then $f=H^b$ and $g=H^a$.
\end{proof}


\section{Main Result}

Let $A=\R [x,y]\cong\R^{[2]}$ and $B=\C\otimes_{\R}A=\C [x,y]\cong\C^{[2]}$. The goal of this section is to prove the following. 

\begin{theorem}\label{main} Define $Q\in A$ by $Q=x^2+y^2-1$. Given $P\in A$, if $A/PA\cong_{\R}A/QA$, then there exist $f,g\in A$ such that 
$A=\R [f,g]$ and $P=f^2+g^2-1$.
\end{theorem}

The proof of this theorem is based on the classification of closed $\C^*$-embeddings in $\C^2$ found in 
\cite{Cassou-Nogues.Koras.Russell.09, Koras.Palka.Russell.16,Koras.Palka.ppt}. 
We show that, for almost every $\C^*$-embedding in their classification, the induced real form of the representative polynomial embedding is an embedding of 
$\R^*$ in $\R^2$. 
There is only one exceptional case, and in this case, the induced embedding of $\Sp^1$ in $\R^2$ is equivalent to the standard embedding. An important  distinction between $\R^*$ and $\Sp^1$ is that the coordinate ring of $\R^*$ has nontrivial units, whereas the units of the coordinate ring of $\Sp^1$ are trivial.

The authors of \cite{Cassou-Nogues.Koras.Russell.09, Koras.Palka.Russell.16,Koras.Palka.ppt} distinguish three types of closed $\C^*$-embeddings in $\C^2$: Those with a very good asymptote, those with a good asymptote, and the sporadic embeddings. These three cases are dealt with in 
{\it Prop.\,\ref{prop1}},  {\it Prop.\,\ref{prop2}}  and {\it Prop.\,\ref{prop3}}, respectively. 


\subsection{One Very Good Asymptote} 

See \cite{Cassou-Nogues.Koras.Russell.09}, Thm.\,8.2\,(i).

\begin{proposition}\label{prop1}
Suppose that $B=\C [u,v]$ and that $P\in B'$ is of one of the following two forms. 
\begin{itemize}
\item [(i.1)] $P(u,v)=v^a-(uv^k+g(v))^b$, 
with $a,b\ge1$ and $\gcd(a,b)=1$, $k\ge 1$, $g(0)=1$ and $g$ otherwise arbitrary of degree at most $k-1$. 
\medskip
\item [(i.2)] $P(u,v)=1-v^{b-a}(uv^{k-1}+g(v))^b$, 
with $b>a\ge 1$, $\gcd(a,b)=1$, $k\ge 1$, $g$ arbitrary of degree at most $k-2$.
\end{itemize} 
\medskip
If $P\in A$ and $(A/PA)^*=\R^*$, then $P$ is of form {\rm (i.1)} with $b=k=1$, and $P$ defines the standard embedding of $\Sp^1$ in $\R^2$. 
\end{proposition}

\begin{proof} First consider the case $[v]=1$ in the polar group $\Pi (A)$. In this case, $v=\omega\alpha$ for $\omega\in\C^*$ and $\alpha\in A'$. 
By {\it Prop.\,\ref{AMS}}, there exists $\beta\in A$ such that $A=\R [\alpha ,\beta]$. So we may assume that $x=\beta$ and $y=\alpha$. 
Since $B=\C [u,y]=\C [x,y]$, it follows that $u=\lambda x + \mu (y)$ for $\lambda\in\C^*$ and $\mu (y)\in\C [y]$. Therefore, form (i.1) becomes
\[
P(u,v)=P(\lambda x+\mu (y), \omega y) = ry^a-(sxy^k+h(y))^b \quad (r,s\in\R^*\, ,\,\, h\in\R[y])
\]
and form (i.2) becomes:
\[
P(u,v)=P(\lambda x+\mu (y), \omega y)=1-ry^{b-a}(sxy^{k-1}+h(y))^b \quad (r,s\in\R^*\, ,\,\, h\in\R[y])
\]
Since $P\in A$, we see that, in each case, the image of $y$ in $A/PA$ is a non-constant invertible function, meaning that $(A/PA)^*\ne\R^*$. 
Therefore, $[v]\ne 1$.

Consider form (i.2). By {\it Lemma\,\ref{P-form}}, we must have $b-a=1$ and $b=1$, which gives a contradiction. 
Therefore, $P(u,v)$ cannot be of form (i.2).

Consider form (i.1). Write $P=vF-1$ for $F\in B'$. By {\it Lemma\,\ref{P-form}}, $FB=\bar{v}B$. If $F=\lambda\bar{v}$ for $\lambda\in\C^*$, 
and if $v=v_1+iv_2$ for $v_1,v_2\in A$, then:
\[
P=\lambda v\bar{v}-1=\lambda (v_1^2+v_2^2)-1\in A \implies \lambda\in\R^*
\]
By {\it Lemma\,\ref{deg}}, it follows that:
\[
2\deg v=\deg P=\max\{ a\deg v,\deg u+bk\deg v\} \implies 2\deg v>bk\deg v \implies b=k=1
\]
We have thus have:
\[
P=v^a-uv-1=v(v^{a-1}-u)-1 \implies F=v^{a-1}-u \implies (u-v^{a-1})B=\bar{v}B
\]
Therefore:
\[
B=\C [u,v]=\C [u-v^{a-1},v]=\C [\bar{v},v] = \C [v_1,v_2]
\]
By {\it Prop.\,\ref{AMS}}, $A=\R [v_1,v_2]$. 
\end{proof}


\subsection{One Good Asymptote} 

See \cite{Cassou-Nogues.Koras.Russell.09}, Thm.\,8.2\,(ii).

\begin{proposition}\label{prop2}
Suppose that $B=\C [u,v]$ and that $P\in B'$ is of one of the following five forms. 
\begin{itemize}
\item [(ii.1)] $v^kP=(v+F^s)^p-F^{sp+1}$ and $F=uv^k+g(v)$, 
where $s,p,k\ge 1$; $g$ is a polynomial of degree at most $k-1$ uniquely determined by the condition that $g$ is a polynomial and $g(0)=1$.
\medskip
\item [(ii.2)] $v^kP=(v+F^s)^p-F^{sp-1}$ and $F=uv^k+g(v)$, 
where $s,p,k\ge 1$, $sp\ge 2$; $g$ is a polynomial of degree at most $k-1$ uniquely determined by the condition that $g$ is a polynomial and $g(0)=1$.
\medskip
\item [(ii.3)] $v^kP=v-16v^2+4vF-8vF^2+F^3-F^4$ for $F=uv^k+g(v)$, where $k\ge 1$, and $g$ is a polynomial of degree at most $k-1$ uniquely determined by the condition that $g$ is a polynomial and $g(0)=1$.
\medskip
\item [(ii.4)] $v^{k-1}P=(1+vF^{s+1})^pF-1$ for $F=uv^{k-1}+g(v)$, where $s,p,k\ge 1$, and $g$ is a polynomial of degree at most $k-2$ uniquely determined by the condition that $g$ is a polynomial.
\medskip
\item [(ii.5)] $v^{k-1}P=(1+vF^{s+1})^p-F$ for $F=uv^{k-1}+g(v)$, where $s,p,k\ge 1$, and $g$ is a polynomial of degree at most $k-2$ uniquely determined by the condition that $g$ is a polynomial.
\end{itemize} 
\medskip
If $P\in A$, then $(A/PA)^*\ne\R^*$. 
\end{proposition}

\begin{proof} Assume, to the contrary, that $(A/PA)^*=\R^*$. Let $G\in B$ be such that $F=vG+1$. 
\medskip

\noindent {\it Form (ii.1).} 
By {\it Lemma\,\ref{cusp}}, there exists $h\in B/PB$ such that $F\equiv h^p$ and $v+F^s\equiv h^{sp+1}$ modulo $P$. 
Therefore:
\[
v\equiv h^{sp}(h-1) \implies 
h^p\equiv vG+1\equiv h^{sp}(h-1)G+1 \implies h^p(1-h^{(s-1)p}(h-1)G)\equiv 1
\]
Consequently, $h$ is a unit modulo $P$, which implies that $F$ is a unit modulo $P$. Since $F\not\in B^*=\C^*$, {\it Lemma\,\ref{unit-degree}} implies:
\[
2\deg F \ge \deg P=sp\deg F + \deg u > sp\deg F \implies s=p=1
\]
We thus have:
\[
v^kP=v+F-F^2=v-vFG \implies v^{k-1}P=1-FG
\]
By {\it Lemma\,\ref{trivial-units}}, $\lambda\bar{F}=G+PL$ for some $L\in B$. However, the equalities $\deg P=\deg F+\deg u$, $\deg \bar{F}=\deg F$ and $\deg G=\deg F-\deg v$ imply $\deg P>\deg \bar{F}>\deg G$, thus precluding the existence of the equation
$\lambda\bar{F}=G+PL$. Therefore, $P$ cannot be of form (ii.1).
\medskip

\noindent {\it Form (ii.2).} Reasoning as in the case of form (ii.1), we find that $2\deg F\ge \deg P$. In this case, $\deg P=(sp-1)\deg F+\deg u$, and it follows that $sp=2$. 

Assume that $s=2$ and $p=1$. Then $v^kP=v+F(F-1)$ implies that $v^{k-1}P=1+FG$. By {\it Lemma\,\ref{trivial-units}}, $\lambda\bar{F}=G+PN$ for some $N\in B$. 
However, the equalities $\deg P=\deg F+\deg u$, $\deg \bar{F}=\deg F$ and $\deg G=\deg F-\deg v$ imply $\deg P>\deg \bar{F}>\deg G$, thus precluding the existence of the equation
$\lambda\bar{F}=G+PN$. Therefore, $s=1$ and $p=2$. 

If $k\ge 2$, then $G=vH-2$ for some $H\in B$, and:
\[
v^kP=(v+F)^2-F \implies v^{k-1}P=v+2F+F(vH-2) \implies v^{k-2}P=1+FH
\]
By {\it Lemma\,\ref{trivial-units}}, there exist $\lambda\in\R^*$ and $M\in B$ with $\lambda\bar{F}=H+PM$. 
However, the equalities $\deg P=\deg F+\deg u$, $\deg \bar{F}=\deg F$ and $\deg H=\deg F-2\deg v$ imply $\deg P>\deg \bar{F}>\deg H$, thus precluding the existence of the equation
$\lambda\bar{F}=H+PM$. Therefore, $k=1$. 

We have $vP=(v+F)^2-F$ and $F=uv+1$, which means that $P=v+2F+uF$. 
Define $h\in B$ by $h=v+F$. Then $P=h+(1+u)F$. Modulo $P$, we have:
\[
h(1+u)\equiv h+uh\equiv h+uv+uF\equiv h+(F-1)-v-2F\equiv h-v-F-1\equiv 1
\]
Therefore, $1+u$ is a unit modulo $P$. By {\it Lemma\,\ref{unit-degree}}, 
\[
2\deg u+\deg v=\deg P\le 2\deg (1+u)=2\deg u
\]
a contradiction. Therefore, $P$ cannot be of form (ii.2). 
\medskip

\noindent{\it Form (ii.3).} If $H=v^{-1}(F^2+4v-F)=uF+4$, then
\[
vH(F^2+4v)=(F^2+4v-F)(F^2+4v)=F^4-F^3+8vF^2-4vF+16v^2=v-v^kP 
\]
which implies:
\[
H(F^2+4v)=1-v^{k-1}P
\]
Therefore, $H$ is a unit modulo $P$. Since $H\not\in B^*=\C^*$, {\it Lemma\,\ref{unit-degree}} implies:
\[
4\deg u+3k\deg v=\deg P\le 2\deg H=4\deg u+2k\deg v
\]
a contradiction. Therefore, $P$ cannot be of form (ii.3). 
\medskip

\noindent {\it Form (ii.4).} By definition, $F$ is a unit modulo $P$, but $F\not\in B^*=\C^*$. Therefore, by {\it Lemma\,\ref{unit-degree}}, $\deg P\le 2\deg F$. However, 
\[
\deg P=p\deg (vF^{s+1})+\deg F-(k-1)\deg v = p\deg (vF^{s+1})+\deg u>2\deg F
\]
which gives a contradiction. 
Therefore, $P$ cannot be of form (ii.4). 
\medskip

\noindent {\it Form (ii.5).} Write $v^{k-1}P=FQ+1$ for $Q\in B$. By {\it Lemma\,\ref{trivial-units}}, there exists $H\in B$ and $\lambda\in\R^*$ such that $\lambda\bar{F}=Q+PH$. 
However, the equalities
\[
\deg P=\deg (vF^{s+1})^p-\deg F+\deg u\,\, ,\,\, \deg Q=\deg (vF^{s+1})^p-\deg F\,\, ,\,\, \deg\bar{F}=\deg F
\]
imply $\deg P>\deg Q>\deg \bar{F}$, thus precluding the existence of the equation
$\lambda\bar{F}=Q+PM$. Therefore, $P$ cannot be of form (ii.5). 

In conclusion, $(A/PA)^*\ne\R^*$ whenever $P\in A$ and $P$ has one of the forms (ii.1)-(ii.5). 
\end{proof}


\subsection{Sporadic Embeddings}

See \cite{Koras.Palka.Russell.16}. Two known families of sporadic embeddings of $\C^*$ in $\C^2$ are parametrized as follows, where the second family has only one member. 

\begin{enumerate}
\item $X=t^{2n}(t^2+t+\frac{1}{2})$ and $Y=t^{-2n-4}(t^2-t+\frac{1}{2})$, where $n\in\Z$, $n\ge 1$. 
\medskip
\item $X=t^4(t^2+t+\frac{2}{3})$ and $Y=t^{-8}(t^2-t+\frac{1}{3})$
\end{enumerate}
In the first case, note that, if $F=4(XY-1)$, then $F=t^{-4}$. Therefore:
\[
 XF^{n+1}=t^{-2n-4}(t^2+t+{\textstyle\frac{1}{2}})=Y+2t^{-2n-3} \implies {\textstyle\frac{1}{2}}(XF^{n+1}-Y)=t^{-2n-3}
\]
We thus obtain the relation:
\[
(Y-XF^{n+1})^4=16F^{2n+3}
\]
Since $\gcd (4,2n+3)=1$ and $\gcd (Y-XF^{n+1},F)=1$ as polynomials in $B$, it follows that this is a prime relation. 

In the second case, let $F=-3(XY-1)$, $G=XF+\frac{4}{9}$ and $H=4Y-3F^2$. Then $H=t^{-6}$ and $3(X-G^2)H-3(1-FG)=t^{-2}$. This gives the relation:
\[
H=27((X-G^2)H+FG-1)^3
\]
Consider polynomials $f,g,h,p\in B=\C [x,y]=\C^{[2]}$ defined by:
\[
f=-3(xy-1)\,\, ,\,\, g=xf+{\textstyle\frac{4}{9}}\,\, ,\,\, h=4y-3f^2 \,\, ,\,\, p=h-27((x-g^2)h+fg-1)^3
\]
By direct calculation (using {\it Maple} for example), we find that $p$ is irreducible, and that the highest-degree homogeneous summand of $p$ is $x^6y^4$. 

\begin{proposition}\label{prop3} Suppose that $B=\C [u,v]$ and that $P\in B'$ is of one of the following two forms. 
\begin{itemize}
\item [(iii.1)] $P=(v-uF^{n+1})^4-16F^{2n+3}$ and $F=4(uv-1)$, where $n\ge 1$.
\medskip
\item [(iii.2)] $P=4v-3F^2-27((u-G^2)H+FG-1)^3$, where $F=-3(uv-1)$ and $G=uF+\frac{4}{9}$.
\end{itemize}
\medskip
If $P\in A$, then $(A/PA)^*\ne\R^*$. 
\end{proposition}

\begin{proof} Assume that $(A/PA)^*=\R^*$. 
\medskip

\noindent {\it Form (iii.1).} By the foregoing discussion, $F$ is a unit modulo $P$. Since $F\not\in B^*=\C^*$, 
{\it Lemma\,\ref{unit-degree}} implies that, if $P\in A$, then $\deg P\le 2\deg F$. But this is impossible, since $\deg P=4\deg (uF^{n+1})$. 
\medskip

\noindent {\it Form (iii.2).} Let $H=4v-3F^2$. By the foregoing discussion, $H$ is a unit modulo $P$. Since $H\not\in B^*=\C^*$, 
{\it Lemma\,\ref{unit-degree}} implies that, if $P\in A$, then $\deg P\le 2\deg H$. But this is impossible, since $\deg P=6\deg u+4\deg v$, while $\deg H=2\deg u+2\deg v$.

Therefore, $(A/PA)^*\ne\R^*$ whenever $P\in A$ and $P$ has one of the forms (iii.1)-(iii.2). 
\end{proof}

This completes the proof of {\it Thm.\,\ref{main}}.


\section{A Remark on $\Sp^1$-Bundles over $\Sp^1$}\label{torus}

Let $R=\R [\Sp^1]$ be the coordinate ring of $\Sp^1$, and let $R=\R [a,b]$, where $a^2+b^2=1$. Define the affine surface $T$ over $\R$ with coordinate ring:
\[
A=R[x,y]/(x^2+y^2-(a+2)^2)
\]
Let $\pi :T\to \Sp^1$ be the surjective morphism induced by the inclusion $\R [a,b]\subset A$. The fiber over $(r,s)\in\Sp^1$ is defined by the quotient ring:
\[
A/(a-r,b-s)= \R [x,y]/(x^2+y^2-(r+2)^2)\cong_{\R}\R [\Sp^1]
\]
So every fiber of $\pi$ is isomorphic to $\Sp^1$, and $T$ is an $\Sp^1$-bundle over $\Sp^1$. 

In $A$, we have $4a=x^2+y^2-a^2-4=x^2+y^2+b^2-5$. Therefore, $A=\R [b,x,y]$, meaning that $T$ admits a polynomial embedding in $\R^3$. The defining polynomial relation is:
\[
\textstyle \frac{1}{4}(x^2+y^2+b^2)^2+b^2=1 \implies (x^2+y^2+b^2)^2=16(x^2+y^2)
\]
We thus recognize $T$ as the surface obtained by revolving the circle $(y-2)^2+b^2=1$ about the $b$-axis. 

\begin{proposition} $T\not\cong_{\R}\Sp^1\times\Sp^1$
\end{proposition}

\begin{proof} (due to Daigle) Consider the integral domain $B=\C\otimes_{\R}A=\C [a,b,u,v]/(uv-(a+2)^2)$, where $u=x+iy$ and $v=x-iy$. Then $B$ has a $\Z$-grading 
$B=\bigoplus_{n\ge 0}B_n$ over $\C [a,b]$ in which $u$ and $v$ are homogeneous, $\deg u=1$ and $\deg v=-1$. In addition: 
\[
B_0=\C [a,b,uv]=\C [a,b]\,\, ,\,\, B_n=u^nB_0 \,\, {\rm for}\,\,  n\ge 1\,\,\ ,\,\, B_n=v^nB_0 \,\, {\rm for}\,\, n\le -1
\]
Any unit of $B$ is homogeneous. Given $w\in B^*$, assume that $\deg w=n\ge 0$ and write $w=u^nc$ for $c\in  B_0$. If $n\ne 0$, then $u\in B^*$, which is absurd since $\dim_{\C}(B/uB)=1$. Therefore, $n=0$. If $\deg w\le 0$, then in the same way we get $\deg w=0$. So $\deg w=0$ in any case. Therefore:
\[
B^*\subset B_0 \implies \C [B^*]=B_0=\C [a,b]
\]

Consider $\C\otimes (\Sp^1\times\Sp^1)=\C^*\times\C^*$, which has coordinate ring $W=\C [s,s^{-1},t,t^{-1}]$. Since $W=\C [W^*]$, it follows that
$B\not\cong_{\C}W$. Therefore, $A\not\cong_{\R}\R [\Sp^1\times\Sp^1]$.
\end{proof}


\noindent {\bf Question.} Does the real algebraic torus $\Sp^1\times\Sp^1$ admit a polynomial embedding in $\R^3$ ?
\medskip


\bibliographystyle{amsplain}

\noindent \address{Department of Mathematics\\
Western Michigan University\\
Kalamazoo, Michigan 49008}\\
\email{gene.freudenburg@wmich.edu}
\bigskip

\end{document}